\documentclass{amsart}
\usepackage{amssymb,amscd,euscript}

\newcommand\CO{\mathbb{C}}
\newcommand\RE{\mathbb{R}}
\newcommand\ZA{\mathbb{Z}}

\newcommand\GL{\operatorname{GL}}
\newcommand\Tr{\operatorname{Tr}}
\newcommand\Str{\operatorname{Str}}
\newcommand\str{\operatorname{str}}
\newcommand\End{{\mathrm{End}}}
\newcommand\Det{\operatorname{Det}}

\newcommand\id{\operatorname{id}}
\newcommand\res{\operatorname{res}}
\newcommand\im{\operatorname{im}}
\newcommand\D{{\EuScript{D}}}
\newcommand\E{{\EuScript{E}}}
\newcommand\F{{\EuScript{F}}}
\newcommand\LL{{\EuScript{L}}}
\newcommand\OO{{\EuScript{O}}}

\newcommand\al{\alpha}
\newcommand\gam{\gamma}
\newcommand\lam{\lambda}
\newcommand\Gam{\mathit{\Gamma}}

\newcommand\PI{\mathit{\Pi}}
\newcommand\Om{\mathit{\Omega}}

\newcommand\bra{\langle}
\newcommand\ket{\rangle}
\newcommand\medwedge{\mbox{\fontsize{12pt}{0pt}\selectfont $\wedge$}}
\newcommand\dbar{\bar\partial}
\newcommand\ii{\sqrt{-1}}
\newcommand\Nab{\mathbb{A}}
\newcommand\A{{A}}

\theoremstyle{plain}
\newtheorem{theorem}{Theorem}[section]
\newtheorem{thm}[theorem]{Theorem}

\newtheorem{prop}[theorem]{Proposition}
\newtheorem{cor}[theorem]{Corollary}
\theoremstyle{definition}
\newtheorem{defn}[theorem]{Definition}

\begin{document}

  \begin{flushright}
{\tt arXiv:1001.3212v2[math.DG]}\\
revised: April, 2010
  \end{flushright}

\vspace{1cm} 

\title[Analytic Torsion of $\ZA_2$-graded Elliptic Complexes]
{Analytic Torsion of $\ZA_2$-graded Elliptic Complexes}

\author{Varghese Mathai}
\address{Department of Mathematics, University of Adelaide,
Adelaide 5005, Australia}
\email{mathai.varghese@adelaide.edu.au}

\author{Siye Wu}
\address{Department of Mathematics, University of Colorado,
Boulder, CO 80309, USA and
Department of Mathematics, University of Hong Kong, Pokfulam Road, Hong Kong} 
\email{swu@math.colorado.edu\qquad swu@maths.hku.hk}

\begin{abstract}
We define analytic torsion of $\ZA_2$-graded elliptic complexes as an element
in the graded determinant line of the cohomology of the complex, generalizing
most of the variants of Ray-Singer analytic torsion in the literature.
It applies to a myriad of new examples, including flat superconnection
complexes, twisted analytic and twisted holomorphic torsions, etc. 
The definition uses pseudo-differential operators and residue traces.
We also study properties of analytic torsion for $\ZA_2$-graded elliptic
complexes, including the behavior under variation of the metric.
For compact odd dimensional manifolds, the analytic torsion is independent
of the metric, whereas for even dimensional manifolds, a relative version
of the analytic torsion is independent of the metric.
Finally, the relation to topological field theories is studied.
\end{abstract}



\thanks{V.M.\ and S.W.\ are supported in part by the Australian Research
Council.
S.W.\ is supported in part by the Research Grants Council of Hong Kong.
We thank the referee for useful comments.}

\maketitle

\section*{Introduction}

In \cite{MW}, we investigated the analytic torsion for the twisted de Rham
complex $(\Om^\bullet(X,\E),d_\E+H\wedge\cdot)$, where $\E$ is a vector bundle
with a flat connection $d_\E$ and $H$ is a closed differential form of odd
degree on a closed compact oriented manifold $X$. 
The novel feature of our construction was the necessary use of
pseudo-differential operators and residue traces in defining the torsion.
When $X$ is odd dimensional, we showed that it was independent of the choice
of metric.
In this paper, we generalize this construction, by defining analytic torsion
for an arbitrary $\ZA_2$-graded elliptic complex as an element in the graded
determinant line of the cohomology of the complex.
The definition again uses pseudo-differential operators and residue traces.
We also study properties of analytic torsion for $\ZA_2$-graded elliptic
complexes, including its behavior under variation of the metric.
For compact odd dimensional manifolds, the analytic torsion is independent
of the metric, whereas for even dimensional manifolds, only a relative
version of the analytic torsion is independent of the metric.
 
We specialize this construction to several new situations where the analytic
torsion can be defined.
This includes the case of flat superconnection complexes and the analytic
torsion of the twisted Dolbeault complex 
$(\Om^{0,\bullet}(X,\E),\bar\partial_{\E}+H\wedge\cdot)$, where $\E$ is 
a holomorphic vector bundle and $H$ is a $\bar\partial$-closed differential
form of type $(0,{\rm odd})$ on a closed connected complex manifold $X$.
When $H$ is zero, this was first studied by Ray and Singer in \cite{RS3}.
Although the definition depends on a choice of Hermitian metric, we deduce
from our general theory that a relative version of torsion, defined as a
ratio of the twisted holomorphic torsions, is independent of the metric.
(Of course they do depend on the complex structure.) 
Twisted holomorphic torsion is defined in several natural situations including,
for Calabi-Yau manifolds, or whenever there is a holomorphic gerbe.
 
Finally, we explain how twisted analytic torsion appears in topological field
theory with a twisted abelian Chern-Simons action functional.

For a more detailed literature review on analytic torsion and its variants,
we refer to the introduction in \cite{MW}.

We briefly summarize the contents of the paper. 
\S\ref{sect:z2} is on the definition of $\ZA_2$-graded elliptic complexes. 
\S\ref{sect:defn} provides the definition of the analytic torsion of a
$\ZA_2$-graded elliptic complex as an element in the graded determinant
line of the cohomology of the complex.
\S\ref{sect:prop} contains functorial properties of the analytic torsion.
\S \ref{sect:inv} estabilshes the invariance of the analytic torsion under
deformation of metrics in the odd dimensional case. 
\S\ref{sect:rel} shows the invariance of the relative analytic torsion under
deformation of metrics in the even dimensional case.
\S\ref{sect:superconn} contains the definition and properties of analytic
torsion of flat superconnections.
\S\ref{sect:dolbeault} contains the definition and properties of the analytic
torsion of twisted Dolbeault complexes.
\S\ref{sect:tft} relates the twisted analytic torsion to topological field
theories.

\section{$\ZA_2$-graded elliptic complexes}\label{sect:z2}

Let $X$ be a smooth closed manifold of dimension $n$ and
$\E=\E^{\bar0}\oplus\E^{\bar1}$, a $\ZA_2$-graded vector bundle over $X$.
(We use $\bar k$ to denote the integer $k$ modulo $2$.)
Suppose $D\colon\Gam(X,\E)\to\Gam(X,\E)$ is an elliptic differential operator
which is odd with respect to the grading in $\E$ and satisfies $D^2=0$. 
Then $D$ is of the form $D={\quad\;\;D_{\bar1}\choose D_{\bar0}\quad\;\;}$
on $\Gam(X,\E)=\Gam(X,\E^{\bar0})\oplus\Gam(X,\E^{\bar1})$, where
$D_{\bar0}\colon\Gam(X,\E^{\bar0})\to\Gam(X,\E^{\bar1})$ and
$D_{\bar1}\colon\Gam(X,\E^{\bar1})\to\Gam(X,\E^{\bar0})$ are differential
operators such that $D_{\bar1}D_{\bar0}=0$ and $D_{\bar0}D_{\bar1}=0$.
Furthermore,
\[ \cdots\to\Gam(X,\E^{\bar0})\stackrel{D_{\bar0}}{\longrightarrow}
\Gam(X,\E^{\bar1})\stackrel{D_{\bar1}}{\longrightarrow}\Gam(X,\E^{\bar0})
\stackrel{D_{\bar0}}{\longrightarrow}\Gam(X,\E^{\bar1})\to\cdots   \]
is a $\ZA_2$-graded elliptic complex, which we denote by $(\E,D)$ for short.
Its cohomology groups are
\[  H^{\bar0}(X,\E,D)=\ker D_{\bar0}/\im D_{\bar1}, \quad
    H^{\bar1}(X,\E,D)=\ker D_{\bar1}/\im D_{\bar0}.  \]
It follows from the Hodge theorem for elliptic complexes as will be explained
shortly that they are finite dimensional.
We call 
\[ b_{\bar0}(X,\E,D)=\dim H^{\bar0}(X,\E,D),\quad
b_{\bar1}(X,\E,D)=\dim H^{\bar1}(X,\E,D)           \]
the Betti numbers of the $\ZA_2$-graded elliptic complex.
Its index or Euler characteristic is
$\chi(X,\E,D)=b_{\bar0}(X,\E,D)-b_{\bar1}(X,\E,D)$.

We choose a Riemannian metric $g$ on $X$ and an Hermitian form of
type $h={h_{\bar0}\quad\;\;\choose\quad\;\;h_{\bar1}}$ on
$\E=\E^{\bar0}\oplus\E^{\bar1}$.
Then there is a scalar product $\bra\cdot,\cdot\ket$ on $\Gam(X,\E)$.
The Laplacian $L=D^\dagger D+DD^\dagger$ on 
$\Gam(X,\E)=\Gam(X,\E^{\bar0})\oplus\Gam(X,\E^{\bar1})$ is, in graded
components, $L={L_{\bar0}\quad\;\;\choose\quad\;\;L_{\bar1}}$, where
\[ L_{\bar0}=D_{\bar0}^\dagger D_{\bar0}+D_{\bar1}D_{\bar1}^\dagger,\quad
L_{\bar1}=D_{\bar1}^\dagger D_{\bar1}+D_{\bar0}D_{\bar0}^\dagger.  \]
They are self-adjoint elliptic operators with positive-definite leading
symbols.
By the Hodge theorem for elliptic complexes, one has
\[ H^{\bar0}(X,\E,D)\cong\ker L_{\bar0},\quad
   H^{\bar1}(X,\E,D)\cong\ker L_{\bar1}.      \]
By ellipticity, these spaces are finite dimensional, and hence $b_{\bar0}$,
$b_{\bar1}$ are finite.
Let $K(t,x,y)=\Big(\substack{K_{\bar0}(t,x,y)\qquad\qquad\\
\qquad\qquad K_{\bar1}(t,x,y)}\Big)$, where $t>0$, $x,y\in X$, be the kernel
of $e^{-tL}={e^{-tL_{\bar0}}\qquad\quad\choose\qquad\quad e^{-tL_{\bar1}}}$.
Suppose the order of $L$ (or that of $L_{\bar0}$ and $L_{\bar1}$) is $d>1$.
By Lemma~1.7.4 of \cite{Gil}, when restricted to the diagonal, the heat kernel
has the asymptotic expansion
\[    K(t,x,x)\sim\sum_{l=0}^\infty t^{\frac{2l-n}{d}}a_l(x),  \]
where $a_l(x)=\Big(\substack{a_{l,\bar0}(x)\qquad\quad\\ \qquad\quad
a_{l,\bar1}(x)}\Big)$ can be computed locally as an combinatorial expression
in the jets of the symbols.
We set $a_{\frac{n}{2}}(x)=0$ if $n$ is odd.
Denote by $a_{\frac{n}{2}}=\Big(\substack{a_{\frac{n}{2},\bar0}\quad\quad\;\\
\quad\quad\;a_{\frac{n}{2},\bar1}}\Big)$ the operator acting on $\Gam(X,\E)$
point-wisely by $a_{\frac{n}{2}}(x)=\Big(\substack{a_{\frac{n}{2},\bar0}(x)
\qquad\quad\\ \qquad\quad a_{\frac{n}{2},\bar1}(x)}\Big)$.
Then the index density of the elliptic complex is $\str a_{\frac{n}{2}}(x)$
and the index is $\chi(X,\E)=\Str(a_{\frac{n}{2}})$.
Here and subsequently, $\str$ is the point-wise supertrace whereas $\Str$ is
the supertrace taken on the space of sections. 

\section{Definition of the analytic torsion}\label{sect:defn}
We generalize the construction in \S2 of \cite{MW}.
Recall that the zeta-function of a semi-positive definite self-adjoint
operator $A$ (whenever it is defined) is
\[  \zeta(s,A)=\Tr'A^{-s}, \]
where $\Tr'$ stands for the trace restricted to the subspace orthogonal to
$\ker(A)$.
If $\zeta(s,A)$ can be extended meromorphically in $s$ so that it is
holomorphic at $s=0$, then the zeta-function regularized determinant of $A$ is
\[  \Det' A=e^{-\zeta'(0,A)}.    \]
If $A$ is an elliptic differential operator of order $d$ on a compact manifold
$X$ of dimension $n$, then $\zeta(s,A)$ is holomorphic when
$\mathrm{Re}(s)>n/d$ and can be extended meromorphically to the entire complex
plane with possible simple poles at $\{\frac{n-l}{m},l=0,1,2,\dots\}$ only
\cite{Se}.
Moreover, the extended function is holomorphic at $s=0$ and therefore
the determinant $\Det'A$ is defined for such an operator.

We return to the setting of the $\ZA_2$-graded elliptic complex $(\E,D)$ in
\S\ref{sect:z2}.
As in \cite{MW}, we consider the partial Laplacian 
$D^\dagger D=\Big(\substack{D_{\bar0}^\dagger D_{\bar0}\qquad\;\\
\qquad\;D_{\bar1}^\dagger D_{\bar1}}\Big)$.

\begin{prop}
For $k=0,1$, $\zeta(s,D_{\bar k}^\dagger D_{\bar k})$ is holomorphic in
the half plane for $\mathrm{Re}(s)>n/d$ and extends meromorphically to $\CO$
with possible simple poles at $\{\frac{n-l}{d},l=0,1,2,\dots\}$ and possible
double poles at the negative integers only, and is holomorphic at $s=0$.
\end{prop}

\begin{proof}
Let $P={P_{\bar0}\quad\;\choose\quad\;P_{\bar1}}$ be the projection onto the
closure of $\im D^\dagger=\im D^\dagger_{\bar0}\oplus\im D^\dagger_{\bar1}$.
As $DD^\dagger$ and $L$ are equal and invertible on (the closure of) $\im D$,
we have
\[ P=D^\dagger(DD^\dagger)^{-1}D=D^\dagger L^{-1}D, \]
which implies that $P$ (and hence $P_{\bar0}$, $P_{\bar1}$) is
a pseudo-differential operator of order $0$.
Moreover, for $k=0,1$,
\[ \zeta(s,D_{\bar k}^\dagger D_{\bar k})=\Tr(P_{\bar k}L_{\bar k}^{-s}). \]
By general theory \cite{GrSe,Gr06}, $\zeta(s,D_{\bar k}^\dagger D_{\bar k})$
is holomorphic in the half plane $\mathrm{Re}(s)>n/d$ and extends
meromorphically to $\CO$ with possible simple poles at $\{\frac{n-l}{d}$,
$l=0,1,2,\dots\}$ and possible double poles at the negative integers only.
The Laurent series of $\zeta(s,D_{\bar k}^\dagger D_{\bar k})$ at $s=0$ is
\[  \Tr(P_{\bar k}L_{\bar k}^{-s})=
\frac{c_{-1}(P_{\bar k},L_{\bar k})}{s}+c_0(P_{\bar k},L_{\bar k})
+\sum_{l=1}^\infty c_l(P_{\bar k},L_{\bar k})\,s^l.   \]
Here $c_{-1}(P_{\bar k},L_{\bar k})=\frac{1}{d}\res(P_{\bar k})$,
where $\res(P_{\bar k})$ is known as the non-commutative residue
of $P_{\bar k}$ \cite{Wo,Gui}.
Since $P_{\bar k}$ is a projection, $\res(P_{\bar k})=0$ \cite{Wo,BrLe,Gr03}.
Therefore $\zeta(s,D_{\bar k}^\dagger D_{\bar k})$ is regular at $s=0$.
\end{proof}

The scalar product on $\Gam(X,\E^{\bar k})$ restricts to one on the null space
of the Laplacian, $\ker(L_{\bar k})\cong H^{\bar k}(X,\E,D)$.
For $k=0,1$, let $\{\nu_{\bar k,i}\}_{i=1}^{b_{\bar k}}$ be an oriented
orthonormal basis of $H^{\bar k}(X,\E,D)$ and let 
$\eta_{\bar k}=\nu_{\bar k,1}\wedge\cdots\wedge\nu_{\bar k,b_{\bar k}}$,
the unit volume element.
Then $\eta_{\bar0}\otimes\eta_{\bar1}^{-1}\in\det H^\bullet(X,\E,D)$.

\begin{defn}
The analytic torsion of the $\ZA_2$-graded elliptic complex $(\E,D)$ is
\[ \tau(X,\E,D)=(\Det'D_{\bar0}^\dagger D_{\bar 0})^{1/2}
                (\Det'D_{\bar1}^\dagger D_{\bar1})^{-1/2}
           \eta_{\bar0}\otimes\eta_{\bar1}^{-1}\in\det H^\bullet(X,\E,D).  \]
\end{defn}

\section{Functorial properties of the analytic torsion}\label{sect:prop}

We summarize some properties of the analytic torsion of $\ZA_2$-graded
elliptic complexes, generalizing those of the Ray-Singer torsion \cite{RS}
and of the torsion of the twisted de Rham complex \cite{MW}.
We omit the proofs as they are similar.

Suppose $X$ is a compact, closed, oriented Riemannian manifold and
$\E_1,\E_2$ are two $\ZA_2$-graded Hermitian vector bundles over $X$.
Then $\E_1\oplus\E_2$ is also a $\ZA_2$-graded vector bundle with
$(\E_1\oplus\E_2)^{\bar k}=\E_1^{\bar k}\oplus\E_2^{\bar k}$ for $k=0,1$.
If $(\E_1,D_1)$ and $(\E_2,D_2)$ are two $\ZA_2$-graded elliptic complexes
on $X$, then so is the direct sum $(\E_1\oplus\E_2,D_1\oplus D_2)$ defined
in the obvious way.
We have the following

\begin{prop}
Under the natural identification of determinant lines,
\[  \tau(X,\E_1\oplus\E_2,D_1\oplus D_2)
    =\tau(X,\E_1,D_1)\otimes\tau(X,\E_2,D_2).    \]
\end{prop}

Now suppose $p\colon X\to X'$ is a covering of compact, closed, oriented
manifolds (with finite index).
Choose a Riemannian metric on $X'$, which pulls back to one on $X$.
Let $\E\to X$ be a $\ZA_2$-graded Hermitian vector bundle.
Then the vector bundle $p_*\E\to X'$ defined by
$(p_*\E)_{x'}=\bigoplus_{x\in p^{-1}(x')}\E_x$ (for $x'\in X'$) is also
$\ZA_2$-graded and has an induced Hermitian form.
There is a natural isometry $\Gam(X,\E)\cong\Gam(X',p_*\E)$.
If $D$ is a differential operator on $\Gam(X,\E)$, the operator $p_*D$ on
$\Gam(X',p_*\E)$ given by the above isomorphism is a differential operator
on $X'$.
If $(\E,D)$ is a $\ZA_2$-graded elliptic complex, then so is $(p_*\E,p_*D)$.
We have

\begin{prop}
Under the natural identification of determinant lines,
\[  \tau(X,\E,D)=\tau(X',p_*D).    \]
\end{prop}

Finally, suppose $X_i$ ($i=1,2$) are closed, oriented Riemannian manifolds
and $\E_i\to X_i$ ($i=1,2$) are $\ZA_2$-graded Hermitian vector bundles.
Denote by $\pi_i\colon X_1\times X_2\to X_i$ ($i=1,2$) the projections.
Set $\E_1\boxtimes\E_2=\pi_1^*\E_1\otimes\pi_2^*\E_2$; it is also a
$\ZA_2$-graded vector bundle with
$(\E_1\boxtimes\E_2)^{\bar0}=\pi_1^*\E_1^{\bar0}\otimes\pi_2^*\E_2^{\bar0}
\oplus\pi_1^*\E_1^{\bar1}\otimes\pi_2^*\E_2^{\bar1}$ and
$(\E_1\boxtimes\E_2)^{\bar1}=\pi_1^*\E_1^{\bar0}\otimes\pi_2^*\E_2^{\bar1}
\oplus\pi_1^*\E_1^{\bar1}\otimes\pi_2^*\E_2^{\bar0}$.
If $(\E_1,D_1)$ and $(\E_2,D_2)$ are two $\ZA_2$-graded elliptic complexes,
then so is $(\E_1\boxtimes\E_2,D_1\boxtimes D_2)$, where the operator
$D_1\boxtimes D_2$ acts on $\Gam(X_1\times X_2,\E_1\boxtimes\E_2)$ according to
\[  (D_1\boxtimes D_2)(\pi_1^*s_1\otimes\pi_2^*s_2)
    =\pi_1^*(D_1s_1)\otimes\pi_2^*s_2
     +(-1)^{|s_1|}\pi_1^*s_1\otimes\pi_2^*(D_2s_2)     \]
for any $s_i\in\Gam(X_i,\E_i)$, $i=1,2$.
We have

\begin{prop}
Under the natural identification of determinant lines,
\[  \tau(X_1\times X_2,\E_1\boxtimes\E_2,D_1\boxplus D_2)
    =\tau(X_1,\E_1,D_1)^{\otimes\chi(X_2,\E_2,D_2)}\otimes
     \tau(X_2,\E_2,D_2)^{\otimes\chi(X_1,\E_1,D_1)}.      \]
\end{prop}

\section{Invariance of the torsion under deformation of metrics:
the odd dimensional case}\label{sect:inv}

We note that the operator $D$, or $D_{\bar0}$ and $D_{\bar1}$ in the
$\ZA_2$-graded elliptic complex $(\E,D)$ are not dependent on the metric.
However, the corresponding partial Laplacians $D^\dagger D$ or
$D_{\bar0}^\dagger D_{\bar0},D_{\bar1}^\dagger D_{\bar1}$ do depend on metric,
and therefore a priori, so does the analytic torsion $\tau(X,\E,D)$. 
In this section, we study the dependence of the analytic torsion on the
metrics and prove that for closed, oriented odd-dimensional manifolds $X$,
the torsion $\tau(X,\E,D)$ is independent of the choice of metric.

Suppose we change the metric $g$ on $X$ and the Hermitian form $h$ on $\E$ 
to $g_u$ and $h_u$, respectively, along a path parameterized by $u\in\RE$.
Although the torsion $\tau(X,\E,D)$ is an element of the determinant line
$\det H^\bullet(X,\E)$, its variation
\[ \frac{\partial}{\partial u}\log\tau(X,\E,D)=
   \tau(X,\E,D)^{-1}\frac{\partial}{\partial u}\tau(X,\E,D)   \]
makes sense as a function of $u$.
For any $u$, the Hermitian structure $\bra\cdot,\cdot\ket_u$ on $\Gam(X,\E)$
is related to the undeformed one by
\[ \bra\,\cdot\,,\cdot\,\ket_u=\bra\Gam_u(\cdot),\cdot\,\ket  \]
for some invertible operator
$\Gam_u={\Gam_{\bar0}\quad\;\,\choose\quad\;\,\Gam_{\bar1}}$.
Let $\al_u=\Gam^{-1}_u\frac{\partial}{\partial u}\Gam_u$.
We have the following

\begin{thm}\label{thm:inv}
Under the above deformation of $g$ and $h$, we have
\[ \frac{\partial}{\partial u}\log\tau(X,\E,D)
   =\Str\big(\al\,a_{\frac{n}{2}}\big). \]
In particular, the above is zero if $\dim X=n$ is odd.
In this case, the analytic torsion $\tau(X,\E,D)$ is independent of the
choice of metric.
\end{thm}

\begin{proof}
We generalize the proof of Lemma~3.1 in \cite{MW}.
The adjoint of $D$ with respect to $\bra\cdot,\cdot\ket_u$ is
$D^\dagger_u=\Gam_u^{-1}D^\dagger\Gam_u$.
Its variation is given by
\[ \frac{\partial}{\partial u}D^\dagger_u=-[\al_u,D^\dagger_u].  \]
In graded components, this is
\[ \frac{\partial D_{\bar0}}{\partial u}=-\al_{\bar1}D_{\bar0}+D_0\al_{\bar0},
  \quad
 \frac{\partial D_{\bar1}}{\partial u}=-\al_{\bar0}D_{\bar1}+D_1\al_{\bar1}, \]
where $\al_{\bar0}=\Gam_{\bar0}^{-1}\frac{\partial}{\partial u}\Gam_{\bar0}$,
$\al_{\bar1}=\Gam_{\bar1}^{-1}\frac{\partial}{\partial u}\Gam_{\bar1}$.
Following \cite{RS,MW}, we set
\begin{align*}
f(s,u)&=\int_0^\infty t^{s-1}\Str(e^{-tD^\dagger D}P)\,dt         \\
      &=\Gamma(s)\big(\zeta(s,D_{\bar0}^\dagger D_{\bar0})
        -\zeta(s,D_{\bar1}^\dagger D_{\bar1})\big),
\end{align*}
omitting the dependence on $u$ on the right-hand side.
Then, as $P\frac{\partial P}{\partial u}=0$,
\begin{align*}
\frac{\partial f}{\partial u} 
&=\int_0^\infty\!t^{s-1}\Str\Big(t[\al,D^\dagger]De^{-tD^\dagger D}
  +Pe^{-tD^\dagger D}\frac{\partial P}{\partial u}\Big)\,dt           \\
&=\int_0^\infty\!t^{s-1}\Str\Big(t\al\big(e^{-tD^\dagger D}D^\dagger D
  +e^{-tDD^\dagger}DD^\dagger\big)
  +e^{-tD^\dagger D}P\frac{\partial P}{\partial u}\Big)dt             \\
&=\int_0^\infty\!t^s\Str\big(\al e^{-tL}L\big)\,dt                    \\
&=-\int_0^\infty\!t^s\frac{\partial}{\partial t}
  \Str\big(\al(e^{-tL}-Q)\big)\,dt                                    \\
&=s\int_0^\infty\!t^{s-1}\Str\big(\al(e^{-tL}-Q)\big)\,dt.
\end{align*}
By the asymptotic expansion of $\Str(\al e^{-tL})$ as $t\downarrow 0$,
$\int_0^1 t^{s-1}\Str(\al e^{-tL})\,dt$ has a possible first order pole
at $s=0$ with residue $\Str(\al a_{\frac{n}{2}})$.
On the other hand, because of the exponential decay of 
$\Str\big(\alpha(e^{-tL}-Q)\big)$ for large $t$, 
$\int_1^\infty t^{s-1}\Str\big(\al(e^{-tL}-Q)\big)\,dt$
is an entire function in $s$.
So
\[ \frac{\partial f}{\partial u}\Big|_{s=0}
   =-\Str\big(\al(a_{\frac{n}{2}}-Q)\big)     \]
is finite and hence
\[ \frac{\partial}{\partial u}\big(\zeta(0,D_{\bar0}^\dagger D_{\bar0})
 -\zeta(0,D_{\bar1}^\dagger D_{\bar1})\big)=0.   \]
Since
\[ \log\Big(\frac{\Det'D_{\bar 0}^\dagger D_{\bar 0}}
                 {\Det'D_{\bar 1}^\dagger D_{\bar 1}}\Big)
=-\lim_{s\to0}\Big[f(s,u)-\frac{1}{s}\big(\zeta(0,D_{\bar0}^\dagger D_{\bar0})
 -\zeta(0,D_{\bar1}^\dagger D_{\bar1})\big)\Big],   \]
we get
\[ \frac{\partial}{\partial u}\log\Big(\frac{\Det'D_{\bar 0}^\dagger
                D_{\bar 0}}{\Det'D_{\bar 1}^\dagger D_{\bar 1}}\Big)
   =\Str\big(\al(a_{\frac{n}{2}}-Q)\big).               \]
Finally, along the path of deformation, the volume elements $\eta_{\bar0}$,
$\eta_{\bar1}$ can be chosen so that (cf.~Lemma~3.3 of \cite{MW})
\[ \frac{\partial}{\partial u}(\eta_{\bar0}\otimes\eta_{\bar1}^{-1})
   =-\frac{1}{2}\Str(\al Q)\,\eta_{\bar0}\otimes\eta_{\bar1}^{-1}.      \]
The results then follow.
\end{proof}

When the elliptic complex is the ($\ZA$-graded) de Rham complex of 
differential forms with values in a flat vector bundle, the variation
of the torsion can be integrated to an anomaly formula \cite{BZ}.

\section{Invariance of relative torsion under deformation of metrics:
the even dimensional case}\label{sect:rel}

When $n=\dim X$ is even, the torsion does depend on the metrics $g$ on $X$
and $h$ on $\E$ (Theorem~\ref{thm:inv}).
However, we will prove that the relative analytic torsion defined below is
independent of the choice of metric.

We first explain extension by flat bundles.
Let $\PI=\pi_1(X)$ be the fundamental group of $X$ and
$\rho\colon\PI\to\GL(m,\CO)$, a representation of $\PI$.
Then $\rho$ determines a flat bundle $\F_\rho$ over $X$ given by
\[\F_\rho=\big(\widetilde X\times\CO^m\big)/\sim,\qquad
     (x\gam,v)\sim(x,\rho(\gam) v),    \]
where $\widetilde X$ is the universal cover of $X$.
Smooth sections of $\F_\rho$ are smooth maps $s\colon\widetilde X\to\CO^m$ that
are $\PI$-equivariant, i.e., $s\circ\gam=\rho(\gam)s$ for all $\gam\in\PI$.
We want to extend $D$ to an action on the sections of
$\E_\rho=\E\otimes\F_\rho$.
Since $D$ is a differential operator, it lifts to the universal cover
$\widetilde X$ as a $\PI$-periodic operator $\widetilde D\colon
\Gam(\widetilde X,\widetilde\E)\to\Gam(\widetilde X,\widetilde\E)$,
where $\widetilde\E$ is the pull-back of $\E$ to $\widetilde X$.
By tensoring with the identity operator on $\CO^m$, we can extend it to
$\widetilde D\colon\Gam(\widetilde X,\widetilde\E\otimes\CO^m)\to 
\Gam(\widetilde X,\widetilde\E\otimes\CO^m)$.
Since for any $\PI$-equivariant section
$s\in\Gam(\widetilde X,\widetilde\E\otimes\CO^m)$,
\[ (\widetilde Ds)\circ\gam=\widetilde D(s\circ\gam)
   =\widetilde D(\rho(\gamma)s)=\rho(\gamma)(\widetilde Ds),   \]
the operator $\widetilde D$ descends to a differential operator
$D_\rho\colon\Gam(X,\E_\rho)\to\Gam(X,\E_\rho)$.
If $(\E,D)$ is a $\ZA_2$-graded elliptic complex, then so is
$(\E_\rho,D_\rho)$.

Now suppose $X$ is a closed, compact, oriented Riemannian manifold and
$\E$ is a $\ZA_2$-graded Hermitian vector bundle.
Let $\rho_1,\rho_2$ be unitary representations of $\PI$ of the same
dimension $m$. 
Then the flat bundles $\F_{\rho_i}$ are Hermitian bundles and so are
$\E_{\rho_i}=\E\otimes\F_{\rho_i}$ ($i=1,2$).
Furthermore, if $(\E,D)$ is a $\ZA_2$-graded elliptic complex as in
\S\ref{sect:z2}, then so are $(\E_{\rho_i},D_{\rho_i})$ for $i=1,2$. 

\begin{defn}
The {\em relative analytic torsion} is the quotient
\[ \tau(X,\E_{\rho_1},D_{\rho_1})\otimes\tau(X,\E_{\rho_2},D_{\rho_2})^{-1}
   \in\det H^\bullet(X,\E_{\rho_1},D_{\rho_1})\otimes
      \det H^\bullet(X,\E_{\rho_2},D_{\rho_2})^{-1}.    \]
\end{defn}

To show its independence of the metric, let $K_{\rho_i}(t,x,y)$ denote,
for $i=1,2$, the heat kernel of the Laplacians
$L_{\rho_i}=D_{\rho_i}^\dagger D_{\rho_i}+D_{\rho_i}D_{\rho_i}^\dagger$.
Since the Hermitian bundles $\E_{\rho_1}$ and $\E_{\rho_2}$, together
with the differential operators $D_{\rho_1}$ and $D_{\rho_2}$ are locally
identical, the difference in the two heat kernels, when restricted to the
diagonal, is exponentially small for small $t$.
More precisely, we have 

\begin{prop}\label{prop:rel}
In the notation above, there are positive constants $C,C'$ such that as
$t\to 0$, one has for all $x\in X$,
\[ |K_{\rho_1}(t,x,x)-K_{\rho_2}(t,x,x)|\le
   Ct^{-n/d}\exp[-C't^{-\frac{d}{d-1}}],  \]
where $d$ is the order of the Laplacians.
\end{prop}

\begin{proof}
Let $\widetilde K(t,x,y)$ denote the heat kernel of the Laplacian
$\widetilde L=\widetilde D^\dagger\widetilde D
+\widetilde D\widetilde D^\dagger$ on $\widetilde X$.
Then by the Selberg principle, one has for $x,y\in\widetilde X$,
\[ K_{\rho_j}(t,\bar x,\bar y)=
   \sum_{\gam\in\PI}\widetilde K(t,x,y\gam)\rho_j(\gam),   \]
where $\bar x\in X$ stands for the projection of $x\in\widetilde X$.
It follows that 
\[  K_{\rho_1}(t,\bar x,\bar y)-K_{\rho_2}(t,\bar x,\bar y)
    =\sum_{\gam\in\PI\setminus\{1\}}\widetilde K(t,x,y\gam)
     (\rho_1(\gam)-\rho_2(\gam)).       \]
Since $\rho_i$ ($i=1,2$) are unitary representations, one has
\[ |K_{\rho_1}(t,\bar x,\bar y)-K_{\rho_2}(t,\bar x,\bar y)|\le
   \sum_{\gam\in\PI\setminus\{1\}}2|\widetilde K(t,x,y\gamma)|.  \]
The off-diagonal Gaussian estimate for the heat kernel on $\widetilde X$ is
\cite{BrSu}
\[  |\widetilde K(t,x,y)|\le C_1t^{-n/d}
       \exp\Big[-C_2\Big(\frac{d(x,y)}{t}\Big)^{\frac{d}{d-1}}\Big],   \]
where $d(x,y)$ is the Riemannian distance between $x,y\in\widetilde X$.
Therefore 
\[ |K_{\rho_1}(t,\bar x,\bar x)-K_{\rho_2}(t,\bar x,\bar x)|
  \le2C_1t^{-n/d}\sum_{\gam\in\PI\setminus\{1\}}
  \exp\Big[-C_2\Big(\frac{d(x,x\gam)}{t}\Big)^{\frac{d}{d-1}}\Big].   \]
By Milnor's theorem \cite{Mi}, there is a positive constant $C_3$ such that 
$d(x,x\gam)\ge C_3\ell(\gam)$, where $\ell$ denotes a word metric on $\PI$.
Moreover, the number of elements in the sphere $S_l$ of radius $l$ in $\PI$
satisfies $\#S_l\le C_4\,e^{C_5l}$ for some positive constants $C_4,C_5$.
Therefore
\begin{align*}
&\sum_{\gam\in\PI\setminus\{1\}}
 \exp\Big[-C_2\Big(\frac{d(x,x\gam)}{t}\Big)^{\frac{d}{d-1}}\Big]            \\    
\le&\sum_{\gam\in\PI\setminus\{1\}}
 \exp\big[-C'(\ell(\gam)/t)^{\frac{d}{d-1}}\big]                             \\ 
\le&\;\sum_{l=1}^\infty\exp\big[-C'(l/t)^{\frac{d}{d-1}}\big]\,C_4\,e^{C_5l} \\
\le&\;C_4\exp[-C't^{-\frac{d}{d-1}}]\sum_{l=1}^\infty
 \exp\big[-C'(l^{\frac{d}{d-1}}-1)+C_5l\big]                        
\end{align*}
for all $t$ such that $0<t\le1$ for some positive constant $C'$.
Since $\frac{d}{d-1}>1$, the infinite sum over $l$ converges and hence
the result.
\end{proof}

\begin{thm}\label{thm:rel}
Let $X$ be a closed oriented manifold of even dimension.
Let $\rho_1,\rho_2$ be unitary representations of $\pi_1(X)$ of the same
dimension.
Then the relative analytic torsion
$\tau(X,\E_{\rho_1},D_{\rho_1})\otimes\tau(X,\E_{\rho_2},D_{\rho_2})^{-1}$
is independent of the choice of metric.
\end{thm}

\begin{proof}
By Theorem~\ref{thm:inv}, under a one-parameter deformation of the metric,
\[ \frac{\partial}{\partial u}\log\tau(X,\E_{\rho_i},D_{\rho_i})
   =\Str\big(\al\,a_{\frac{n}{2}}^{\rho_i})\big)                 \]
for $i=1,2$. 
By Proposition~\ref{prop:rel}, we have
$a_{\frac{n}{2}}^{\rho_1}=a_{\frac{n}{2}}^{\rho_2}$.
Therefore the change in relative torsion is zero.
\end{proof}

\section{Analytic torsion of flat superconnections}\label{sect:superconn}
 
The concept of superconnection was initiated by Quillen, cf.~\cite{Q,MQ,BGV}.
Let $X$ be a smooth manifold and $\F=\F^{\bar0}\oplus\F^{\bar1}$,
a $\ZA_2$-graded vector bundle over $X$.
Then the space $\Om(X,\F)$ of $\F$-valued differential forms has a
$\ZA_2$-grading with
\[ \Om(X,\F)^{\bar0}=\Om^{\bar0}(X,\F^{\bar0})\oplus\Om^{\bar1}(X,\F^{\bar1}),
   \quad
 \Om(X,\F)^{\bar1}=\Om^{\bar0}(X,\F^{\bar1})\oplus\Om^{\bar1}(X,\F^{\bar0}). \]
A {\em superconnection} is a first-order differential operator $\Nab$ on
$\Om(X,\F)$ that is odd with respect to the $\ZA_2$-grading and satisfies
\[ \Nab(\al\wedge s)=d\al\wedge s+(-1)^{|\al|}\al\wedge\Nab s    \]
for any $\al\in\Om(X)$ and $s\in\Om(X,\F)$.
The bundle $\End(\F)$ is also $\ZA_2$-graded and $\Nab$ extends to
$\Om(X,\End(\F))$.
The curvature of the superconnection is
$F_\Nab=\Nab^2\in\Om(X,\End(\F))^{\bar0}$.
It satisfies the Bianchi identity $\Nab F_\Nab=0$.
A superconnection $\Nab$ is of the form $\Nab=\nabla+\A$, where $\nabla$
is a usual connection on $\F$ preserving the grading and
$\A\in\Om(X,\End(\F))^{\bar1}$.
Thus the superconnections form an affine space modeled on the vector space
$\Om(X,\End(\F))^{\bar1}$.

The superconnection is flat if $F_\Nab=0$.
In this case, writing $\Nab={\quad\;\Nab_{\bar1}\choose\Nab_{\bar0}\quad\;}$,
there is a $\ZA_2$-graded elliptic complex
\[  \cdots\to\Om(X,\F)^{\bar 0}\stackrel{\Nab_{\bar0}}{\longrightarrow}
    \Om(X,\F)^{\bar1}\stackrel{\Nab_{\bar1}}{\longrightarrow}\Om(X,\F)^{\bar0}
    \stackrel{\Nab_{\bar0}}{\longrightarrow}\Om(X,\F)^{\bar1}\to\cdots,     \]
We can define the cohomology groups $H^{\bar k}(X,\F,\Nab)$, $k=0,1$.
In fact, this is a special case of \S\ref{sect:z2} with
$\E=\medwedge TX\otimes\F$ and $D=\Nab$.
If $X$ is a closed, compact, oriented Riemannian manifold and $\F$ is an
Hermitian vector bundle, then we can define the analytic torsion of a flat
superconnection as $\tau(X,\F,\Nab)=\tau(X,\E,D)\in\det H^\bullet(X,\F,\Nab)$
with the above choice of $(\E,D)$.
The functorial properties (\S\ref{sect:prop}) and invariance under metric
deformations (\S\ref{sect:inv}, \ref{sect:rel}) hold in this case.

We consider a special case when $\F=\F^{\bar0}$ and $\F^{\bar1}=0$.
Then $\Om(X,\F)^{\bar k}=\Om^{\bar k}(X,\F)$ for $k=0,1$.
A superconnection is of the form $\nabla+\A$, where $\nabla$ is a usual
connection on $\F$ and $\A\in\Om^{\bar1}(X,\End(\F))$.
Suppose $\A$ is of degree $3$ or higher.
Then the superconnection is flat if and only if $\nabla$ is flat and
$\nabla\A+\A^2=0$.
When $\A$ is of the form $\A=H\id_\F$ for some $H\in\Om^{\bar1}(X)$, the above
condition on $\A$ becomes $dH=0$ and the $\ZA_2$-graded elliptic complex is
the twisted de Rham complex $(\Om(X),d+H\wedge\cdot)$.
Its analytic torsion $\tau(X,\F,H)$ was studied in \cite{MW}.
Among other properties, the latter is also invariant under the deformation of
$H$ by an exact form when $X$ is odd dimensional; the rest of the section will
be devoted to generalizing this property to the analytic torsion of flat
superconnections.

We return to the general case of a flat superconnection $\Nab$ over a graded
vector bundle $\F$.
Suppose $G\in\Om(X,\End(\F))^{\bar0}$ is point-wisely invertible.
Then $\Nab'=G^{-1}\Nab G$ is another flat superconnection on $\F$; we say
that $\Nab'$ is gauge equivalent to $\Nab$.
There is an isomorphism of cohomology groups
$H^\bullet(X,\F,\Nab)\cong H^\bullet(X,\F,\Nab')$, and hence of the
corresponding determinant lines, induced by $G$.
Now suppose $\Nab$ is deformed to $\Nab_v$ along a path parameterized by $v$
so that each $\Nab_v$ is gauge equivalent to $\Nab$ via $G_v$.
Let
\[\beta_v=G_v^{-1}\frac{\partial G_v}{\partial v}\in\Om(X,\End(\F))^{\bar0}.\]

\begin{thm}
Under deformation of $\Nab$ by gauge equivalence and the natural
identification of determinant lines, we have
\[ \frac{\partial}{\partial v}\log\tau(X,\F,\Nab)
   =\Str\big(\beta\,a_{\frac{n}{2}}\big). \]
If $\dim X=n$ is odd, then the above is zero.
In this case, the analytic torsion $\tau(X,\F,\Nab)$ is invariant under
gauge equivalence.
\end{thm}

\begin{proof}
Under the deformation, we have
\[ \frac{\partial\Nab}{\partial v}=[\beta,\Nab],\qquad
\frac{\partial\Nab^\dagger}{\partial v}=-[\beta^\dagger,\Nab^\dagger]. \]
The component of $\beta$ in $\Om^{\bar1}(X,\End(\F)^{\bar1})$ does not
contribute to the trace or supertrace, whereas that in
$\Om^{\bar0}(X,\End(\F)^{\bar0})$ is even in the degree of differential forms.
Following the proofs of Lemmas~3.5 and 3.7 of \cite{MW}, we get the desired
variation formula upon a suitable choice of volume elements and identification
of determinant lines under the deformation; the rest follows easily.
\end{proof}

If $\dim X$ is even, a relative version of analytic torsion
(cf.~\S\ref{sect:rel}) is invariant under gauge equivalence.

\section{Analytic torsion of twisted Dolbeault complexes}\label{sect:dolbeault}

Let $X$ be a connected, closed, compact complex manifold and $\F$,
a holomorphic vector bundle over $X$.
Denote by $\Om^{p,q}(X,\F)$ the space of smooth differential $(p,q)$-forms
on $X$ with values in $\F$.
A holomorphic connection on $\F$ can act on $\Om^{p,q}(X,\F)$ and splits
uniquely as $\partial_\F+\dbar_\F$, where
\[ \partial_\F\colon\Om^{p,q}(X,\F)\to\Om^{p+1,q}(X,\F),\quad
   \dbar_\F\colon\Om^{p,q}(X,\F)\to\Om^{p,q+1}(X,\F)           \]
satisfying $\dbar_\F^2=0$.
This yields the Dolbeault complex of differential forms with values in $\F$.

Let $\Om^{p,\bar0}(X,\F)$, $\Om^{p,\bar1}(X,\F)$ be the space of differential
forms that is of degree $p$ in the holomorphic part and of even, odd degree,
respectively, in the anti-holomorphic part.
Consider a differential form $H\in\Om^{0,\bar1}(X)$ that is $\dbar$-closed,
i.e., $\dbar H=0$.
Let $\dbar_{\F,H}=\dbar_\F+H\wedge\cdot\;$.
We call $H$ a holomorphic flux form and $\dbar_{\F,H}$, the Dolbeault operator
twisted by $H$.
Setting $\dbar_{\bar k}= \bar{\partial}_{\F,H}$ acting on
$\Om^{p,\bar k}(X,\E)$ for $k=0,1$, we have
$\dbar_{\bar1}\dbar_{\bar0}=\dbar_{\bar0}\dbar_{\bar1}=0$ and a $\ZA_2$-graded
elliptic complex, which we call the {\em twisted Dolbeault complex}
\[    \cdots\to\Om^{p,\bar0}(X,\F)\stackrel{\dbar_{\bar0}}{\longrightarrow}
\Om^{p,\bar1}(X,\F)\stackrel{\dbar_{\bar1}}{\longrightarrow}\Om^{p,\bar0}(X,\F)
\stackrel{\dbar_{\bar0}}{\longrightarrow}\Om^{p,\bar1}(X,\F)\to\cdots.    \]
We define the {\it twisted Dolbeault cohomology groups} as
\[ H^{p,\bar0}(X,\F,H)=\ker\,\dbar_{\bar0}/\im\,\dbar_{\bar1},\quad
   H^{p,\bar1}(X,\F,H)=\ker\,\dbar_{\bar1}/\im\,\dbar_{\bar0}.  \]
Like in \cite{RW,MW}, if the degree of $H$ is $3$ or higher, there is a
spectral sequence whose $E_2$-terms are $H^{p,\bullet}(X,\F)$ converging to 
$H^{p,\bullet}(X,\F,H)$.
If $H'$ and $H$ differ by a $\dbar$-exact form, then there are natural
isomorphisms $H^{p,\bullet}(X,\F,H')\cong H^{p,\bullet}(X,\F,H)$.

The above construction is the holomorphic counterpart of the twisted de Rham
complex studied in \cite{RW,MW}.
Holomorphic flux forms arise naturally in a number of prominent situations.
Suppose that $X$ is a Calabi-Yau manifold of an odd complex dimension $n$.
Then the canonical bundle of $X$ is trivial, i.e., there is a nowhere zero
section $\Om$ which satisfies $\partial\Om=0$.
Here $H=\overline\Om\in\Omega^{0,n}(X)$ is $\dbar$-closed.
Another example comes from holomorphic gerbes (or holomorphic sheaves of
groupoids).
The 3-curvature of a holomorphic curving on a holomorphic gerbe on a complex
manifold $X$ is a closed holomorphic $3$-form $\Om$ on $X$
(cf.~\cite{Bry}, 5.3.17 part (4)).
Again, $H=\overline\Om\in\Om^{0,3}(X)$ is $\dbar$-closed.

The twisted Dolbeault complex is also a special $\ZA_2$-graded elliptic
complex $(\E,D)$ with
$\E=\medwedge^p(T^{1,0}X)^*\otimes\medwedge^\bullet(T^{0,1}X)^*\otimes\F$
and $D=\dbar_{\F,H}$.
Suppose $X$ is closed and compact.
Given a Riemannian metric on $X$ and an Hermitian form on $\F$, we have the
analytic torsion of the twisted Dolbeault complex (cf.~\S\ref{sect:defn})
\[   \tau_p(X,\F,H)=\tau(X,\E,D)\in\det H^{p,\bullet}(X,F,H)   \]
with the above choice of $(\E,D)$.
It is satisfies the functorial properties in \S\ref{sect:prop}.
Since $X$ is always of even (real) dimension, only a relative version of
the analytic torsion for the twisted Dolbeault complex is independent of
the metric.
We conclude from Theorem~\ref{thm:rel} the following

\begin{cor}
Let $\F$ be a holomorphic vector bundle over a compact complex manifold $X$.
Suppose $H\in\Om^{0,\bar1}(X)$ is $\dbar$-closed.
For two flat bundles on $X$ given by the representations $\rho_1,\rho_2$ of
$\pi_1(X)$ of the same dimension, the relative twisted holomorphic torsion 
$\tau(X,\F_{\rho_1},H)\otimes\tau(X,\F_{\rho_2},H)^{-1}$ is invariant under any
deformation of $H$ by an $\dbar$-exact form, up to natural identification of
the determinant lines.
\end{cor}

For a non-trivial example of twisted holomorphic torsion, consider the compact
Calabi-Yau manifold $T\times M$, where $T$ is a compact complex torus of
dimension $1$ and $M$ is a K3 surface.
Let $\LL=\LL_{u,v}$ be a flat line bundle over $T$ defined by the character
$\chi_{u,v}(m,n)=\exp(2\pi\ii(mu+nv))$, $0\le u,v\le1$, $m,n\in\ZA$.
If $(m,n)\ne(0,0)$, then the Dolbeault cohomology $H^\bullet(T,\LL)$ is
trivial.
Recall the non-trivial holomorphic torsion of $(T,\LL)$ \cite{RS3}
\[  \tau_0(T,\LL)= \left|e^{\pi\ii v^2\tau}
    \frac{\theta_1(u-\tau v,\tau)}{\eta(\tau)}\right|,    \]
where $\tau$ (with $\mathrm{Im}\,\tau>0$) be the complex moduli of $T$.
Here the theta function is defined as 
\[ \theta_1(w,\tau)=-\eta(\tau)e^{\pi\ii(w+\tau/6)}\prod_{k=-\infty}^\infty
                    (1- e^{2\pi\ii(|k|\tau-\epsilon_kw)}),      \]
where $\epsilon_k={\rm sign}\left(k+\frac{1}{2}\right)$ and $\eta(\tau)$ is
the Dedekind eta function.
We still denote by $\LL$ the pull-back of $\LL$ to $T\times M$.
The Dolbeault cohomology groups of $T\times M$ are trivial and so are the
twisted ones. 
Since $\chi(\OO_T(\LL))=0$ and $\chi(\OO_M)=2$, we have \cite{RS3}
\[    \tau_0(T\times M,\LL)=\tau_0(T,\LL)^{\otimes 2}.    \]
Let $H=\bar\al\wedge\bar\lam$, where $\al$ is a holomorphic $1$-form on $T$
and $\lam$ a holomorphic $2$-form on $M$.
By perturbation theory \cite{Fa}, one has,
\[ \tau_0(T\times M,\LL,H)= e^{o(|H|)}\tau_0(T\times M,\LL)
 =e^{o(|H|)}\tau_0(T,\LL)^{\otimes 2},                             \]
where $o(H)\to0$ as $H\to0$.
Therefore $  \tau_0(T\times M,\LL,H)$ non-trivial whenever $|H|$ is
sufficiently small.

\section{Relation to topological field theories}\label{sect:tft}

In \cite{Sch}, a topological field theory of antisymmetric tensor fields
were constructed and the partition function is shown to be equal to the
Ray-Singer analytic torsion.
The metric independence of the torsion is an evidence that the quantized
theory is topological invariant.
In this section, we extend the relation to twisted analytic torsion by
constructing topological field theories that contain a coupling with the
flux form.

Suppose $X$ is a compact, oriented manifold of dimension $n$ and $H$ is
a flux form, a closed differential form of odd degree.
For $k=0$ or $1$, we define a theory whose action is 
\[   S_{\bar k}[B,C]=\int_X B\wedge d_HC,    \]
where $B\in\Om^{\overline{n-k}}(X)$, $C\in\Om^{\bar k}(X)$ are the dynamical
fields.
Since the operator $d_H=d+H\wedge\cdot\,$ is not compatible with the 
$\ZA$-grading, the forms $B,C$ can not be chosen to have fixed degrees. 
Instead, the degrees of $B,C$ have the same parity when $\dim X$ is odd and
opposite parity when $\dim X$ is even.
The classical equations of motion are
\[    d_HC=0,\qquad d_{-H}B=0.   \]
The action $S[B,C]$ and the equations of motion are invariant under a set of
gauge transformations
\[    C\mapsto C+d_HC^{(1)},\qquad B\mapsto B+d_{-H}B^{(1)},    \]
where $B^{(1)},C^{(1)}$ can be any forms whose degrees have opposite parity
with $B,C$, respectively.
The phase space is the space of solutions to the equation of motion modulo the
gauge transformations.
In this case, it is $H^{\overline{n-k}}(X,-H)\oplus H^{\bar k}(X,H)$, expressed
in terms of the de Rham cohomology groups twisted by the fluxes $\pm H$.

To quantize the theory, we consider the partition function
\[   Z_{\bar k}(X,H)=\int\D B\D C\;e^{-S_{\bar k}[B,C]}.           \]
We need to introduce a Riemannian metric on $X$ which determines the
``measures'' $\D B$, $\D C$.
The integration of the transverse parts of $B,C$ yields the determinant
$\Det'(d_H^\dagger d_H)^{-1/2}$ (defined in \S2 of \cite{MW}); that of the zero
modes contributes volume elements on the cohomology groups.
The longitudinal modes of $B,C$ are treated by adding Faddeev-Popov ghost
fields which contribute to determinant factors in the numerator, and there
are secondary and higher ghosts since $B^{(1)},C^{(1)}$ themselves contain
redundancies.

We consider a special case when $\dim X=2l+1$ is odd and $H$ is a top-degree
form (cf.~\S5.1 of \cite{MW}).
If $B,C\in\Om^{\bar1}(X)$, then $B\wedge d_HC=BdC$, and the theory is
equivalent to an untwisted theory.
We now assume that $B,C\in\Om^{\bar0}(X)$.
Then the bosonic determinant from the integration of the transverse modes is
\[  \Det'{d_0^\dagger d_0+H^\dagger H\quad H^\dagger d_{2l}
\choose\quad\;d_{2l}^\dagger H\qquad\;\;d_{2l}^\dagger d_{2l}}^{\!\!-1/2\,}
\prod_{i=1}^{l-1}(\Det'd_{2i}^\dagger d_{2i})^{-1/2},             \]
where $d_i$ is $d$ on $\Om^i(X)$ for $0\le i\le 2l+1$.
The crucial feature in this case is that $H$ does not appear in the gauge
transformations
\[    B\mapsto B+dB^{(1)},\qquad C\mapsto C+dC^{(1)}.    \]
Moreover, we can choose $B^{(1)},C^{(1)}\in\Om^{\bar1}(X)$ to be of degree
$2l-1$ or less.
Further redundancies in $B^{(1)},C^{(1)}$ are described by a hierarchy of
gauge transformations
\[   B^{(i)}\mapsto B^{(i)}+dB^{(i+1)},\qquad
     C^{(i)}\mapsto C^{(i)}+dC^{(i+1)},         \]
where $B^{(i)},C^{(i)}\in\Om^{\bar i}(X)$ are of degree $2l-i$ or less,
for $1\le i\le2l-1$.
The Faddeev-Popov procedure yields the determinant factors
\begin{align*}
&\prod_{i=0}^{2l}\Big[\Det'(d_{2l-i}^\dagger d_{2l-i})
  \Det'(d_{2l-i-2}^\dagger d_{2l-i-2})\cdots\Big]^{(-1)^{i+1}}   \\
&=\prod_{i=0}^l\Det'(d_{2i}^\dagger d_{2i})^{-l/2}
  \prod_{i=0}^{l-1}\Det'(d_{2i+1}^\dagger d_{2i+1})^{(l+1)/2}.
\end{align*}
Taking into account the contribution of the zero modes, the partition
function is
\[ Z_{\bar0}(X,H)=\tau(X,H)^{-1}\otimes\tau(X)^{\otimes(-l)}
   \in\det H^\bullet(X,H)^{-1}\otimes\det H^\bullet(X)^{\otimes(-l)}.    \]
Here $\tau(X)\in\det H^\bullet(X)$ is the classical Ray-Singer torsion
\cite{RS}.
The independence of the partition on the metric indicates that the quantum
theory is also metric independent although it is necessary to use a metric
in the definition.

It would be interesting to construct topological field theories when the
flux form $H$ is of an arbitrary degree, when the manifold has a boundary
\cite{Wu}, and those related to the analytic torsion of other $\ZA_2$-graded
elliptic complexes such as the twisted Dolbeault complex.

\end{document}